\documentclass[11pt]{article}
\usepackage[utf8]{inputenc}
\usepackage[
papersize={5.5in, 8.5in},
left=0.35in,
right=0.35in,
top=0.35in,
bottom=0.5in]{geometry}
\usepackage{graphicx}
\usepackage{amsthm}
\usepackage{rotating}
\usepackage{amsfonts,amssymb,amsmath}
\usepackage[usenames]{color}

\usepackage{array}
\usepackage{float}
\usepackage{booktabs}
\usepackage[dvipsnames,usenames,table,xcdraw]{xcolor}

\newcommand{\gr}{\cellcolor[HTML]{E0E0E0}}

\graphicspath{ {images/} }
\newtheorem*{theorem}{Theorem}
\title{\bf 
The chromatic number of the plane \\ is at least 5 -- a human-verifiable proof
}
\author{\bf 
\textcolor[rgb]{1,0.25,0.6}{Jaan Parts} \\
} 
\date{\normalsize \textcolor[rgb]{1,0.25,0.6}{Kazan, Russia, jaan\_parts@.mail.ru}}

\begin{document}

\maketitle

\pagestyle{empty}
\thispagestyle{empty}

\begin{abstract}
We present a new proof of the known fact that the chromatic number of the plane is at least 5. The main difference of this proof is that it can be verified manually without the help of the computer.
\end{abstract}

\section{Background}

Ever since Marty McFly found himself at the Peabody\footnote{The characters from Robert Zemeckis's film "Back to the Future".} farm (one might say, since 1955) and until recently, I could not understand what could be interesting about the breeding of trees. And then the moment came when I opened my eyes: this is really interesting!

Before I tell you, what made me change my views, we’ll have to go back in time. Not so far, just a couple of years ago, to the time Aubrey de Grey's article \cite{grey} was published, where he proved that the chromatic number of the plane is at least 5.

This broke the ice in a problem that had been frozen since about the time of old man Peabody. The \textit{chromatic number} $\chi$ of the plane is the smallest number of colors required to cover the entire Euclidean plane such that any two of its points at a unit distance have a different color. This coloring is called \textit{proper}. Currently it is known that $\chi\in\{5, 6, 7 \}$.

De Grey's proof is not only the first, but, in our opinion, the best of the known ones. But like all others found thus far, it has an annoying flaw: it cannot be verified without using a computer. Here we make an attempt to fill this gap: that is, we present a proof of the known fact $\chi\ge5$, which can be verified manually in full in a reasonable time.

This paper is organized as follows. Section 2 describes our way through labyrinths of coral caves to the light. Also, we introduce here some useful concepts. In Section 3, we give the general outline of the human-verifiable proof. In Section 4, we take a closer look at the part of the proof that previously required computer verification. In Section 5, we describe the methods used to reduce the number of elementary checks. Section 6 discusses the possibilities for further simplifying the proof.

\section{Ramble on}

Here we discuss some ideas, mostly those that have turned out to be dead ends. Perhaps because we missed some obvious move. 

Initial hopes to simplify de Grey's proof were somehow connected with the idea of minimizing 5-chromatic graphs and their subgraphs thereof for which the number of 4-colorings of some specific vertices is substantially limited.

Thus, for example, graphs are known in which certain pairs of vertices (and hence the points of the plane) at distances $1/3$, $5/3$, $7/3$, $3$ and $\sqrt{11/3}$ are non-monochromatic \cite{par}: that is, they cannot be the same color in any proper 4-coloring. A pair of points at a distance $8/3$, by contrast, is necessarily monochromatic (from here there is only one step to a 5-chromatic graph: it remains to form an isosceles triangle of two such mono-pairs and a single edge). Sets of three points forming an equilateral triangle with side $\sqrt3$, $\sqrt5$, $\sqrt7$, or $1/\sqrt3$ are necessarily non-mono. The triple with side $\sqrt3$ is used in de Grey's proof \cite{grey}; the triple with side $1/\sqrt3$ appears in an alternative proof given by Exoo and Ismailescu \cite{exoo}.

At the moment, all proofs we know require a computer check of a so-called base graph containing more than 200 vertices \cite{par} to verify that it possesses some property of the kind just listed. The task of analyzing all possible variants of coloring such a graph is far beyond human capabilities. However, it is clear that the problems of minimizing graphs and simplifying the proof are in different planes and have different requirements for graphs.

Our first attempt to step further was the idea of going from a non-mono-triple to sets of vertices that would lead to smaller graphs. What if we consider a set of $m$ triples, only some number $n$ of which should have the property of non-mono-chromaticity? For convenience, we will call such a set a hyper-triple with probability $p\ge n/m$.

It can be expected that with decreasing $p$ the required number of vertices of the corresponding graph will decrease. At some point, it will decrease so much that it will be easy to verify the property of non-mono-chromaticity. After that, one can try to go in the opposite direction, combining hyper-triples and gradually increasing $p$ until we get the original non-mono triple with $p=1$.

A trivial example of a hyper-triple with $p\ge1/2$ is a pair of triples with one common vertex and two pairs of vertices connected by an edge. We tried to find examples of non-trivial pairs of triples and got some progress in this direction. So, we considered different sets of concentric triples rotated relative to each other by some angle $\alpha$. For a pair of triples with side $\sqrt3$ and $p\ge1/2$, we found a 181-vertex graph for $\alpha=\pi$ and a 115-vertex graph for $\alpha=\arccos(5/6)$.

\begin{figure}[!b]
\centering
\begin{tabular}{cc}
\includegraphics[scale=0.3]{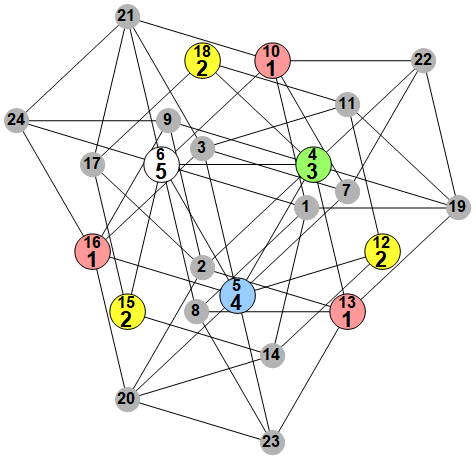} & \includegraphics[scale=0.3]{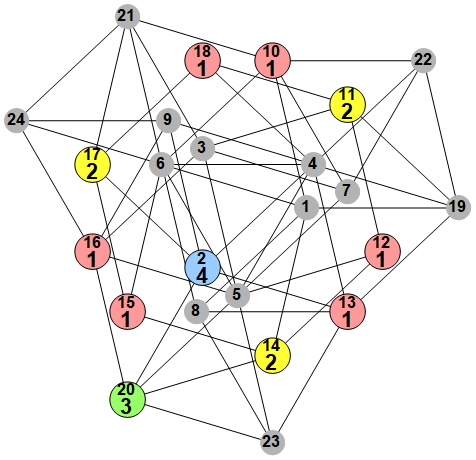}
\end{tabular} \par
\caption{Two variants of the initial coloring of the 24-vertex graph $G_{24}$, all of the three triples of which with side $\sqrt3$ are monochromatic. Vertex numbers are shown in small print, color numbers are shown in large print.}
\label{g24}
\end{figure}

\begin{figure}[!b]
\begin{tabular}{l@{\!}}
\includegraphics[scale=0.38]{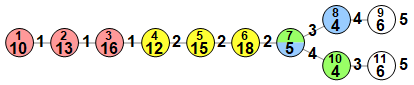} \\
\includegraphics[scale=0.38]{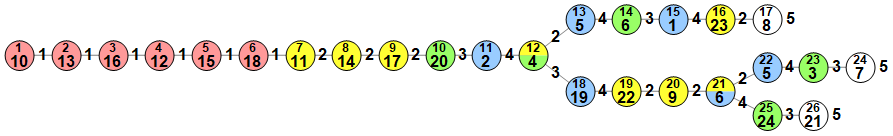}
\end{tabular} \par
\caption{Coloring trees for the graph $G_{24}$. Node numbers are shown in small print, vertex numbers are shown in large print. Color number is indicated on the edges of the tree.}
\label{t24}
\end{figure}

Finally, we discovered the 24-vertex graph shown in Fig.\ref{g24}, in which at least one of the three triples $T_a=\{10, 13, 16\}$, $T_b=\{12, 15, 18\}$, $T_c=\{11, 14 , 17\}$ is non-mono. For such a small graph, it is already quite easy to give a compact proof of this property. If the triples $T_a$ and $T_b$ are mono, then they are colored the same (let it be color 1), since they are both connected to the 3-clique $\{4, 5, 6\}$. If the triple $T_c$ is also mono, then it has a different color (let it be color 2), since it is connected to the triple $T_b$. Now we choose for vertex 20 one of the two remaining colors (let it be color 3) and continue coloring the vertices in the sequence shown in Fig.\ref{t24} (lower image). We thereby build a so-called coloring tree, which pretty quickly leads to a vertex that cannot have any of the four colors. This tree has two branch points. The same figure (upper image) shows a coloring tree for the case where the mono-triples $T_a$ and $T_b$ have different colors. Since all paths of coloring lead to a contradiction, we can conclude that at least one of the triples $T_a$, $T_b$ and $T_c$ is non-mono.

Then one can add one more triple $T_d$, rotated by an angle $\pi$ relative to the triple $T_a$, and consider all possible sets of four triples, three each. Since each such set contains a non-mono-triple, at least two triples of these four must be non-mono: in other words, $p\ge2/4$.

However, we were unable to advance further. If we take a couple of triples at once (the 181- and 115-vertex graphs), then the size of the coloring tree increases significantly. Worse, we were unable to find a way to make more complex hyper-triples from less complex ones. And the number of nodes (colored vertices) of the first tree, which we got for a 221-vertex graph with a non-mono-triple with side $\sqrt3$, was more than 80\,000. It would be difficult to manually check such a large tree.
	
We tried to work with other constructions. Note that in de Grey's proof one can replace non-mono-triples by pairs with distance~2. When coloring a 7-vertex wheel graph (Fig.\ref{g7}), the colorings with two concentric non-mono-triples are exactly those with three concentric pairs (i.e. the diagonals), at least one of which is mono. This allows us to consider different sets of such pairs, starting with, say, a dozen pairs, one of which is mono, and gradually reduce the number of such pairs, moving to more complex graphs. However, we did not get success on this path either.

The fan-like graph obtained by sequentially connecting four 4-vertex rhombus graphs with a common vertex seems tempting too. The distance between the extreme vertices of this graph is $8/3$, that is, a mono-pair can be formed. For three such rhombuses, the distance between the extreme vertices is $\sqrt{11/3}$, which can be a non-mono pair. The only question is what to do with it?

\newpage

\begin{figure}[H]
\centering
\includegraphics[scale=0.3]{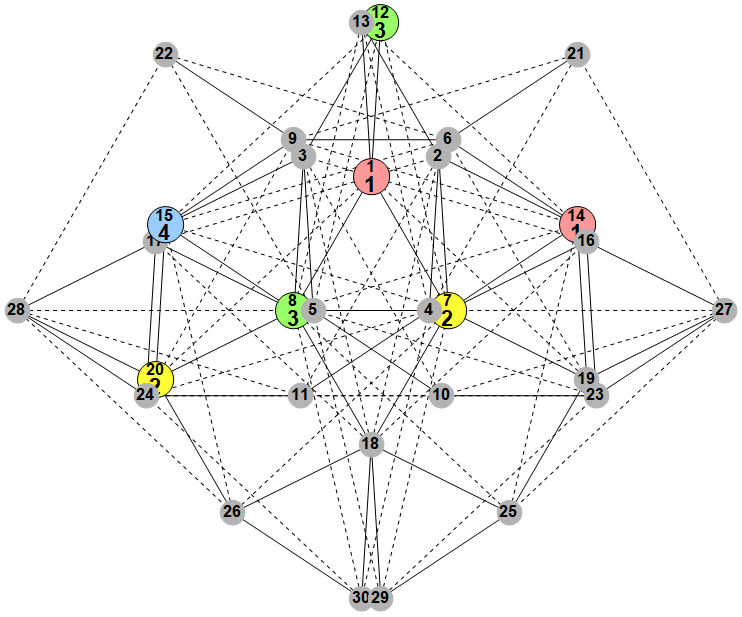}
\caption{Initial coloring of the 30-vertex two-distance graph $G_{30}$ with mono-pair of distance $8/3$ between vertices 14 and 15. Edges of length $\sqrt{11/3}$ are dashed lines, unit-distance edges are solid lines.}
\label{g30}
\end{figure}

\begin{figure}[H]
\centering
\includegraphics[scale=0.38]{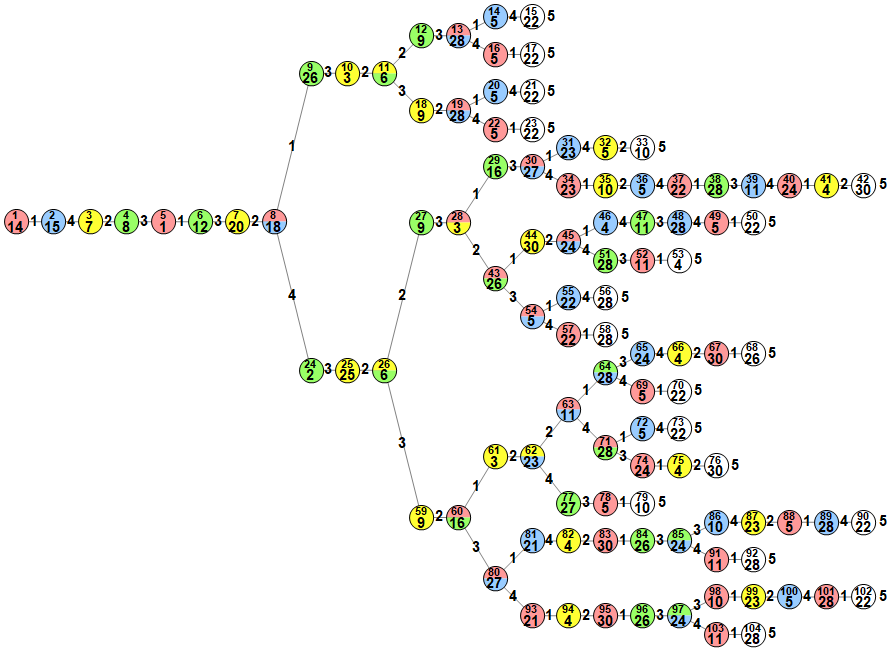}
\caption{Coloring tree for the graph $G_{30}$.}
\label{t30}
\end{figure}

Finally, we went back to the coloring trees. We noticed that if we take a base graph that is not particularly compact, we can often get a tree of noticeably smaller size. This is quite understandable: in this case, one can choose a more optimal coloring sequence. An increase in the symmetry order of the base graph also promised good prospects. Further work, in fact, boiled down to improving the method of growing compact trees. This path gave good results in the end.

It remains to present the proof of $\chi\ge5$ in the form that has taken shape at the moment and describe the method used to minimize the coloring trees. But first, let us demonstrate the effectiveness of this approach with one more example.

Consider a 30-vertex two-distance graph (Fig.\ref{g30}) with 96 edges, 48 each for distances 1 and $\sqrt{11/3}$, and containing a mono-pair $\{14, 15\}$ with a distance $8/3$. This graph can replace the 40-vertex graph appearing in the previously mentioned alternative proof \cite{exoo}. The relatively small size of this graph does not make it convenient for analysis. One can start by coloring the four vertices $\{7, 8, 14, 15\}$, which take on different colors, assuming the pair $\{14, 15\}$ is non-mono. It is also possible, due to symmetry, to fix the color of vertex 1 and two more vertices. But then the number of coloring options grows rapidly, and it becomes difficult to keep them all in mind. The situation changes when we build a tree (see Fig.\ref{t30}).

\section{The fifth element}

\paragraph{Definitions.}
A coloring \textit{tree} is a sequence of assignments of colors to the vertices of a graph, which uses for each vertex all color options not prohibited by prior assignments. Each such assignment is called a tree \textit{node}. Depending on the coloring method and the number of permitted colors, the following types of nodes are distinguished: \textit{root} (with a predetermined color), \textit{branch} (which can be colored in two or more colors), \textit{stem} (which can be colored only in a single color), \textit{end} (which cannot be dyed in any of the colors).
The tree can be represented as a graph or as a list of nodes. A vertex can appear in multiple branches of the same tree.

A coloring \textit{diagram} is a sequence of nodes with all interconnections with previously colored nodes. (The diagram is a generalization of the tree concept and can be thought of as a tree braided with lianas.) In a simplified form, the diagram contains only the necessary interconnections, one per color \textit{port} of a given node.

By \textit{minimizing} a tree, we mean reducing the number of its nodes. 

A \textit{working} graph is the current graph to be colored. 

A \textit{root} graph is a subgraph of a working graph with a limited set of options for coloring vertices up to symmetries (permutations of vertices and colors). Each such option is considered separately, giving the initial coloring of the working graph vertices.

A \textit{base} graph is a graph containing all considered working graphs. The symmetry group of a root graph must be a subgroup of the symmetry group of the base graph. 

Two vertices of the graph belong to the same \textit{orbit} if there exists an automorphism that maps one vertex to another.

\paragraph{Base graph.}
We use base graphs with a symmetry group order 24 and vertex coordinates, which on the complex plane can be represented as
$$\frac{a+b\sqrt{33}+ic\sqrt{3}+id\sqrt{11}}{12},
\;\;a,b,c,d \in \mathbb{Z}, \;a-b+c+d \in 4\mathbb{Z},
$$
so it is convenient \cite{exoo} to denote each vertex by an integer vector $(a, b, c, d)$.

An orbit of the base graph is denoted by one of its vertices. The coordinates of the remaining vertices of the orbit can be obtained by rotations by an angle that is a multiple of $\pi/3$, by reflections relative to the coordinate axes, and by the \textit{conjugation} symmetry, which changes the sign of the elements $b$ and $c$ of the coordinate vector.

\paragraph{Root graph.}
We use, as root graphs, various extensions of 10-vertex Golomb graph that differ in the number $t$ of inner triangles. For convenience, a unified numbering of the vertices of $t$-Golomb graphs is used (Fig.\ref{num}). Note that the conjugation symmetry given by the permutation of the vertices $(3n-1,\: 3n)$, $n\in\{1,2,3,4,5,6\}$, acts on these graphs.

\begin{figure}[!b]
\centering
\begin{tabular}{@{}cccc@{}}
\includegraphics[scale=0.36]{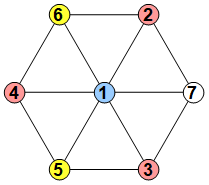} & \includegraphics[scale=0.36]{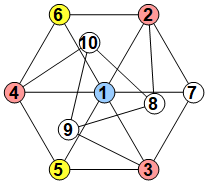} &
\includegraphics[scale=0.36]{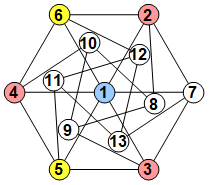} & \includegraphics[scale=0.36]{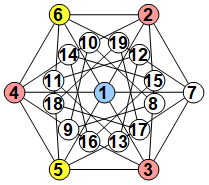}
\end{tabular} \par
\caption{Unified numbering of root vertices (nodes) for the wheel graph, and 1-, 2-, 4-Golomb graphs (from left to right). Vertices 2, 3, 4, highlighted in red, correspond to the original mono-triple.}
\label{num}
\end{figure}

\paragraph{Main idea.}
The essence of the approach is that a tree is created that reflects the sequence of coloring the vertices of a particular graph, and each such sequence ends with a vertex that cannot be colored in any of the colors. Then this tree is used ready-made at one of the stages of the proof. Thus, this stage consists of a series of similar checks. The number of such checks is large, but quite reasonable (in our case, less than 800). Each check consists of several elementary actions, where at the input there are between two and four vertices of different colors, and at the output, a new vertex appears, colored in one of the remaining colors. An elementary check also includes calculating unit distances (according to the number of input vertices).

\vspace{4mm}

\begin{theorem}
{\upshape \cite{grey}} $\chi\ge5$.
\end{theorem}

\begin{proof}

We follow de Grey's proof \cite{grey}. 

Assume that four colors are enough. 

\begin{figure}[!b]
\centering
\begin{tabular}{@{}cccc@{}}
\includegraphics[scale=0.36]{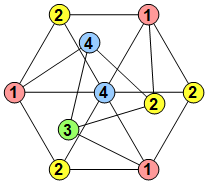} & \includegraphics[scale=0.36]{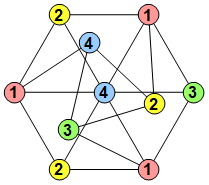} &
\includegraphics[scale=0.36]{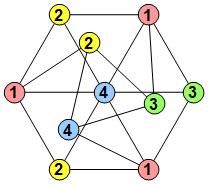} & \includegraphics[scale=0.36]{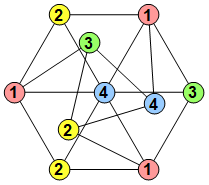}
\end{tabular} \par
\caption{All four non-isomorphic colorings of the Golomb graph containing mono-triple with side $\sqrt3$ highlighted in red.}
\label{gol}
\end{figure}

In the first part, it is proved that three points forming an equilateral triangle with side $\sqrt3$ cannot be colored the same. Assume the opposite, and consider three such points of the same color (let it be color 1). Then, up to a permutation of vertices and colors, there are only two different coloring options for the corresponding 7-vertex wheel graph, and four coloring options for the 10-vertex Golomb graph (Fig.\ref{gol}). (The remaining options for coloring the inner triangle of the Golomb graph are obtained by vertex permutation $(1)(2,3)(4)(5,6)(7)(8,9)(10)$ corresponding to the conjugation symmetry.)
Each of these two graphs can be used as a root graph, and each of the coloring options of the root graph can be associated with a coloring tree of the base graph (shown in Fig.\ref{base}), all branches of which end in vertices that cannot be colored in any of the four colors. It remains to consistently check the coloring of all nodes of this given set of trees. This is the most time-consuming part of the proof, and it is discussed in detail in the next section. Since all possible coloring sequences lead to an uncolorable vertex, the original triple is therefore non-mono.

\begin{figure}[!t]
\centering
\includegraphics[scale=0.34]{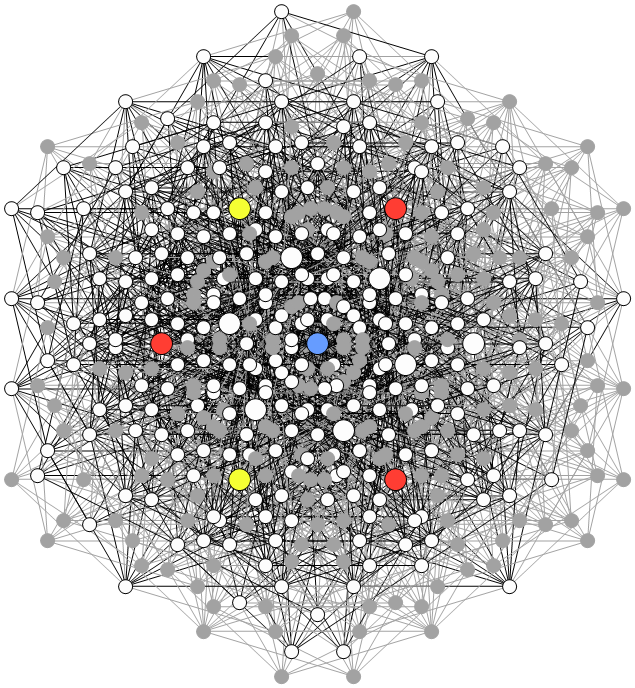}
\caption{The 481-vertex base graph $G_{481}$. Fixed root vertices are shown in color. Other vertices used in the coloring trees are highlighted in white. In total, 268 vertices are used. Unused vertices are shown in gray.}
\label{base}
\end{figure}

In the second part of the proof, based on the non-mono-triples obtained above, we build a construction that cannot be colored with four colors. A 7-vertex wheel graph containing two such non-mono-triples has only two non-isomorphic colorings (Fig.\ref{g7}). Hence it follows, provided that all triples with side $\sqrt3$ are non-mono, that any coloring of a hexagonal lattice formed by edges of unit length, at least in one of the three lattice directions, consists of alternating lattice lines with vertices of only two colors.
Indeed, if each wheel subgraph of the lattice has the coloring shown on the left in Fig.\ref{g7}, then this property holds for all three lattice directions. If a wheel subgraph shown on the right in Fig.\ref{g7} is encountered, then it uniquely determines such coloring of three adjacent lattice lines (the next lattice vertex in the row $\{1,4,1\}$ can only be of color 4, otherwise mono-triples appear), and therefore all other parallel lattice lines are also two-color. It follows that the 7-vertex set obtained by doubling all distances of the wheel graph can be colored in only two colors, and there are only three non-isomorphic coloring options shown on Fig.\ref{g19}.

Now consider two such sets with a common center, connected by edges. There are only two options for coloring the resulting 13-vertex set (see Fig.\ref{g37}), and any pair of vertices at distance 4 is mono. It remains to connect two such mono-pairs by an edge (Fig.\ref{g37} on the right) in order to come to a contradiction with the initial assumption of 4-colorability.


\begin{figure}[!b]
\centering
\begin{tabular}{@{}cc@{}}
\includegraphics[scale=0.36]{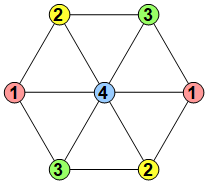} & \includegraphics[scale=0.36]{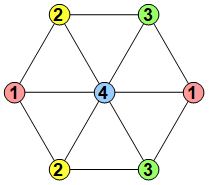}
\end{tabular} \par
\caption{Two non-isomorphic colorings of the wheel graph $G_7$ provided that both triples with side $\sqrt3$ are non-mono.}
\label{g7}
\end{figure}

\begin{figure}[!b]
\centering
\begin{tabular}{@{}ccc@{}}
\includegraphics[scale=0.36]{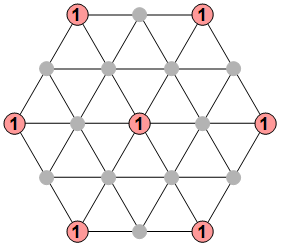} & \includegraphics[scale=0.36]{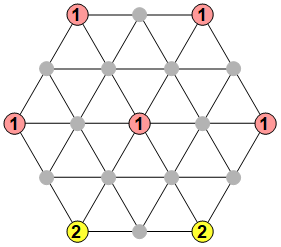} & \includegraphics[scale=0.36]{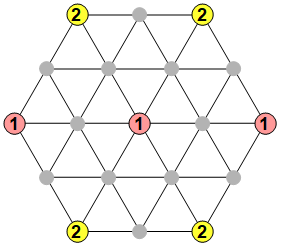}
\end{tabular} \par
\caption{Three non-isomorphic colorings of the part of hexagonal lattice provided that all triples with side $\sqrt3$ are non-mono.}
\label{g19}
\end{figure}

\begin{figure}[!b]
\centering
\begin{tabular}{@{}ccc@{}}
\includegraphics[scale=0.36]{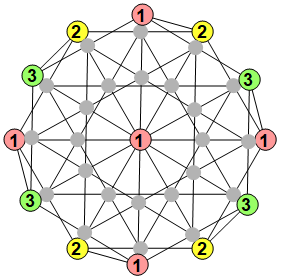} & \includegraphics[scale=0.36]{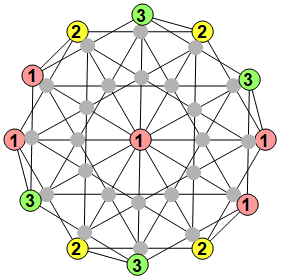} & \includegraphics[scale=0.36]{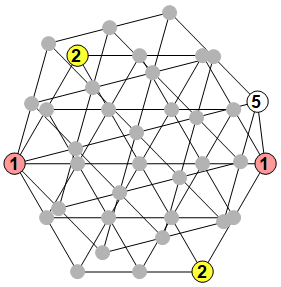}
\end{tabular} \par
\caption{The last two steps of the proof. On the left, two non-isomorphic colorings of two hexagonal lattices are shown, one of which is rotated by $\arccos(7/8)$. On the right is the coloring of the spindle formed by relative rotation of two hexagonal lattices by $\arccos(31/32)$.}
\label{g37}
\end{figure}

\end{proof}

\section{In the light}

Here we consider the first part of the proof in more detail. 

\paragraph{The coloring trees.}
At the initial stage, we use symmetry properties (isomorphic permutations of colors and vertices of the base graph) as much as possible in order to obtain the minimum set of root colorings of maximum length, leading presumably to small trees.

It is possible to stop at the set of colorings $R_a$, $a\in1\dots4$, of 10-vertex Golomb graph (see Fig.\ref{gol} and table \ref{root}), but we added three more vertices to this root graph, which allowed us to slightly reduce the total number of nodes. This gives 36 options for coloring the 13-vertex 2-Golomb graph (6 options for $R_1$ and 10 options each for $R_2$, $R_3$ and $R_4$). Extended colorings are denoted as $R_{ab}$, where one index $b\in0\dots9$ is added to the notation.

\begin{table}[!b]
\caption{Root colorings of a 2-Golomb graph and their correspondence to the roots of small and large trees. Root colorings that lead to large trees are shown gray.}
\label{root}
\smallskip
\small
\centering
\begin{tabular}{@{}c|c|*{10}{|>{\!}c<{\!}}}
\hline
$R_{ab}$ &\multicolumn{1}{r||}{-$b$}  &-0	&-1	&-2	&-3	&-4	&-5	&-6	&-7	&-8	&-9	\\
\hline
  &vertices	&\multicolumn{10}{c}{11\dots13 (colored 143 for $R_{12}$, and 413 for $R_{17}$)}	\\
\cline{2-12}
$a$-	 &  1\dots10	&132	&134	&142	&312	&314	&341	&342	&412	&431	&432	\\
\hline
\hline
1-	 &4 111 222 234	&-	    &$S$	&\gr$L_1$	&-	&$S$	&\gr$L_2$	&-	&\gr$L_2$	&$S$	&-	\\
\hline
2-	 &4 111 223 234	&$S$	&$S$	&\gr$L_3$	&\multicolumn{1}{c}{} &\multicolumn{2}{c}{$L_4$}	&\gr &\gr$L_5$	&$S$	&$S$	\\
\hline
3-	 &4 111 223 342	&$S$	&\gr$L_6$	&$S$	&$S$	&$S$	&$S$	&$S$	&$S$	&\gr$L_7$	&$S$	\\
\hline
4-	 &4 111 223 423	&\multicolumn{3}{c}{} &\multicolumn{4}{c}{\gr$L_8$} &\multicolumn{3}{c}{} \\
\hline
\end{tabular} \par
\end{table}

Further analysis shows that 24 of these 36 root colorings quickly lead to a contradiction (they give trees with up to 8 non-root nodes). Of these, 19 contain a mono-pair 
between the opposite vertices of two inner triangles of 2-Golomb graph, and Fig.\ref{small} (middle image) shows that such a mono-pair is prohibited. The colorings $R_{15}=\{4\,111\,222\,234\,341\}$ and $R_{17}=\{4\,111\,222\,234\,413\}$ are isomorphic, which is ensured by the permutation of colors $(1, 2)(3)(4)$ and rotation by the angle $\pi/3$. The colorings $R_{23}\dots R_{26}$ and $R_{40}\dots R_{49}$ are convenient to combine by discarding some vertices (see table \ref{root}).

As a result, we got nine different root colorings with the number of vertices from 7 to 13, eight of which lead to \textit{large} trees $L_1$\dots$L_8$ (Fig.\ref{large}), and remaining one leads to a \textit{small} tree $S$ (Fig.\ref{small}).

\begin{figure}[!b]
\centering
\begin{tabular}{@{}ccc@{}}
\includegraphics[scale=0.361]{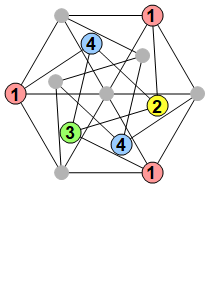} & \includegraphics[scale=0.358]{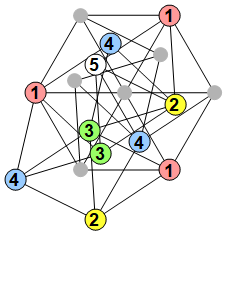} & \includegraphics[scale=0.36]{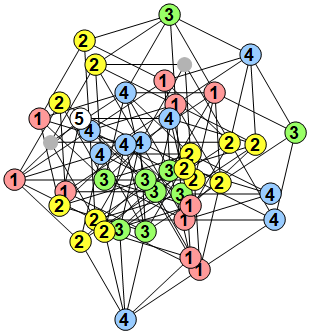}
\end{tabular} \par
\caption{Root graph of small tree $S$ and full graphs corresponding to trees $S$ and $L_2$ (from left to right). The numbers indicate the vertex colors. Note that vertices of the same color 
form monochromatic clusters.}
\label{small}
\end{figure}

\begin{figure}[!b]
\centering
\begin{tabular}{@{}cccc@{}}
\includegraphics[scale=0.36]{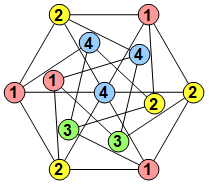} & \includegraphics[scale=0.36]{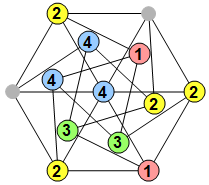} &
\includegraphics[scale=0.36]{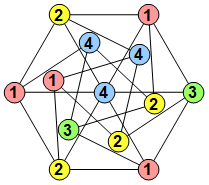} & \includegraphics[scale=0.36]{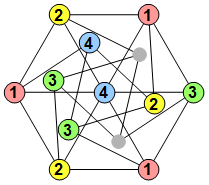} \\
\\
\includegraphics[scale=0.36]{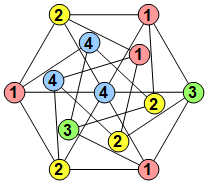} & \includegraphics[scale=0.36]{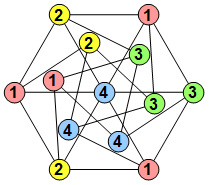} &
\includegraphics[scale=0.36]{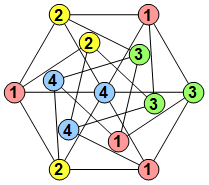} & \includegraphics[scale=0.36]{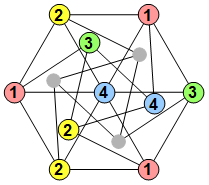} 
\end{tabular} \par
\caption{Root 2-Golomb graphs of large trees $L_{1\dots8}$ (left to right, top to bottom). The numbers indicate the vertex colors. Undefined or unused root vertices are shown in gray.}
\label{large}
\end{figure}

The resulting trees are shown in Fig.\ref{trees} (although here they look more like Mexican cacti). Some parameters of trees are given in the table \ref{param}. In our case, all non-root branch nodes are bifurcations. 
Table \ref{orb} lists orbits of the base graph and a number of orbit vertices used by trees.

\paragraph{A coloring diagram.}
In our case, the problem of convenient data presentation plays an important role. For small trees, a coloring that obviously leads to a contradiction can be easily shown directly on the graph (see Fig.\ref{small}, second image). For large trees, the density of vertices and edges increases, and in addition, branches appear. The same figure (third image) shows a 
graph for the tree $L_2$ with a total of 43 colored vertices and no branches. As can be seen, even in this truncated form, it is not very convenient for analysis. 
To solve this problem and eliminate confusing interconnections, all the information necessary to perform the verification, including the vertex coordinates, is integrated into a so-called coloring diagram and is presented in a tabular form\footnote{All required data can also be found on the Polymath project website \cite{pm16}.}
(see table \ref{nodes}).

Each node of the diagram has its own ordinal number, four color ports and four integer vertex coordinates. The ports are arranged in ascending order of color number from 1 to 4 (for convenience, color coding has also been added, although it is not necessary to use it). Each port lists previously colored neighbor nodes of given color (zero corresponds to the unused colors of the root node). 
The unnumbered (empty) port gives the output color of this node. Branch nodes (with several coloring options) are grayed. End nodes (for which all 4 colors are occupied) are colored white. Frames of special nodes (root, branch, end) are thickened.

\paragraph{Verification notes.}
The nodes of the diagram are checked sequentially. One elementary check in most cases involves checking the unit distance from the vertex corresponding to the current node to three neighboring vertices of the specified nodes, and checking the colors of these vertices.

The most time-consuming operation is calculating distances. In our case, it is enough to take the difference of two vectors of vertex coordinates in the form $(a, b, c, d)$, and make sure that it belongs to one of 30 options\footnote{It is easy to verify that each of these options gives a distance of unit length.}
$(\pm12, 0, 0, 0)$, $(\pm10, 0, 0, \pm2)$, $(\pm6, 0, \pm6, 0)$, $(0, \pm2, \pm2, 0)$, $(\pm5, \pm1,\pm5,\pm1)$, or $(\pm3, \pm1,\pm1,\pm3)$, and in the last two cases, the number of minus signs must be odd.

Additionally, one must make sure that all colorings lie on one of the paths from the root of the tree. On the first pass of a branch node, it can be marked, and after reaching an end node, one can return to a nearest marked branch node, marking all intermediate nodes as prohibited. On the second pass, one can mark this branch node as a stem node, and skip all prohibited nodes.

If we discard the root nodes, then the total number of nodes of nine trees is 787. With some skill, one can spend about half a minute on one elementary check, or about 100 checks per hour. Thus, it is quite possible to handle proof verification in a day. (If only the coloring is checked, then one can keep within a couple of hours.)

\newpage

\begin{figure}[H]
\centering
{
\centering
\begin{tabular}{cccccccccc}
    \includegraphics[scale=0.175]{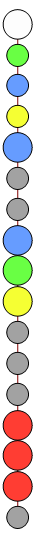} &
    \includegraphics[scale=0.24]{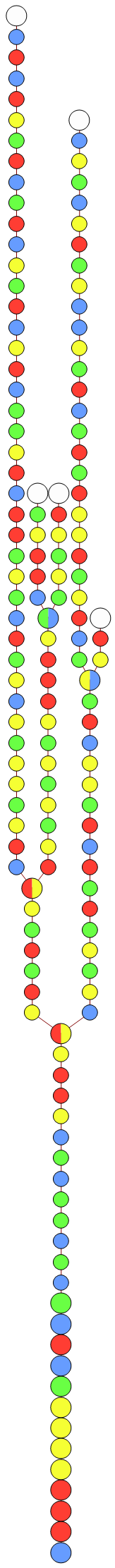} & \includegraphics[scale=0.24]{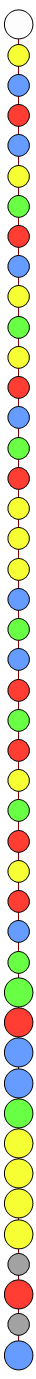} & \includegraphics[scale=0.24]{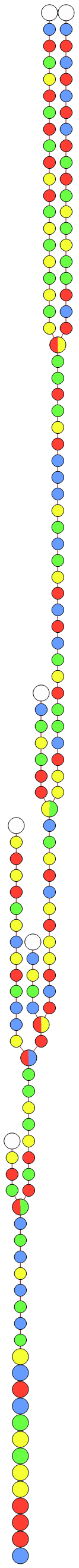} & \includegraphics[scale=0.24]{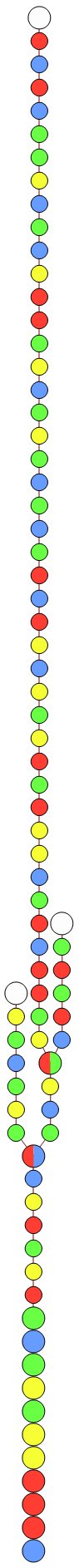} & \includegraphics[scale=0.24]{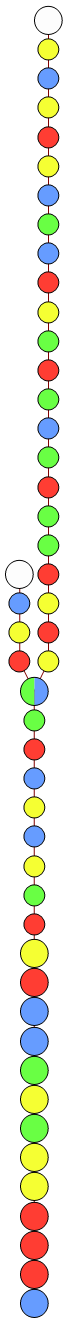} & \includegraphics[scale=0.24]{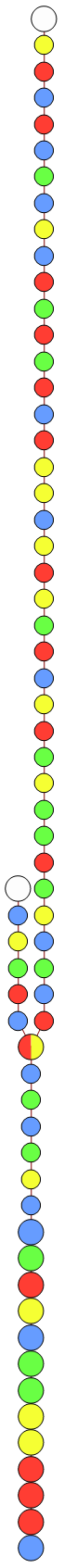} & \includegraphics[scale=0.24]{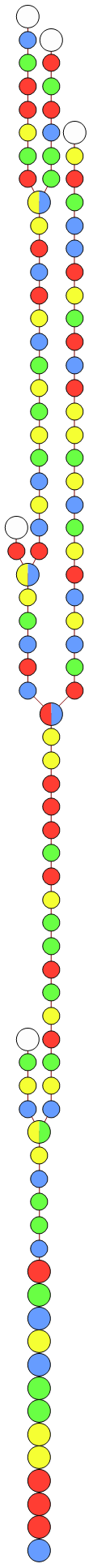} & \includegraphics[scale=0.24]{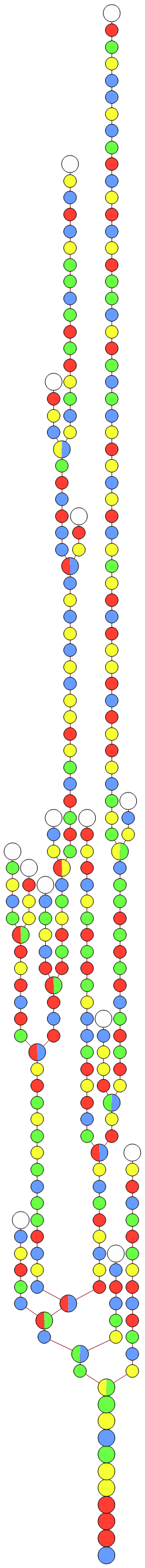} \\
    $S$ & $L_1$ & $L_2$ & $L_3$ & $L_4$ & $L_5$ & $L_6$ & $L_7$ & $L_8$ 
\end{tabular} \par
}
\caption{Trees for all nine non-isomorphic 2-Golomb roots. Special nodes (root, branch, end) are enlarged. The nodes whose color is not specified are shown in gray. End nodes are colored white.}
\label{trees}
\end{figure}

\newpage

\begin{table}[H]
\caption{Parameters of coloring trees (diagrams).}
\label{param}
\smallskip
\small
\centering
\begin{tabular}{@{}l|ccccccccc@{}}
\hline
tree name         &$S$ &$L_1$ &$L_2$ &$L_3$ &$L_4$ &$L_5$ &$L_6$ &$L_7$ &$L_8$\\
\hline
\hline
nodes in total    &11 &147 &43 &144 &80 &49 &65 &105 &247\\
root nodes        &7  &13  &11 &13  &11 &13 &13 &13  &10 \\
end nodes         &1  &5   &1  &6   &3  &2  &2  &5   &14 \\
vertices used     &11 &113 &43 &130 &74 &46 &64 &88  &155\\
\hline
elementary checks &4  &134 &32 &131 &69 &36 &52 &92  &237\\
\hline
\end{tabular} \par
\end{table}

\begin{table}[H]
{\caption{Orbits filling of the base graph $G_{481}$. Root orbits are grayed. }
\label{orb}

{
\centering
\small
\begin{tabular}{>{\!\!}c<{\!\!}|>{\!}c<{\!}|c|c|c<{\!}*{8}{>{\!}c<{\!\!}}}
\hline
orbit & \!orbit\! & \!\!vertex\!\! & \!orbit\! &\multicolumn{9}{c}{number of used vertices in trees}  \\
\cline{5-13}
 & \!\!radii\!\! & \!\!degree\!\! & \!order\! &$S$ &$L_1$ &$L_2$ &$L_3$ &$L_4$ &$L_5$ &$L_6$ &$L_7$ &$L_8$ \\
\hline
\hline

\gr (0,0,0,0)  &0      &30 &1  &0 &1 &1 &1 &1 &1 &1 &1 &1\\
\hline
    (0,0,4,0)  &.5774  &24 &6  &0 &4 &1 &5 &4 &1 &3 &6 &6\\
\gr (12,0,0,0) &1      &23 &6  &3 &6 &4 &6 &6 &6 &6 &6 &6\\
    (0,0,8,0)  &1.1547 &18 &6  &0 &2 &0 &2 &0 &0 &0 &1 &3\\
    (0,0,12,0) &1.7321 &8  &6  &0 &0 &1 &0 &0 &1 &0 &0 &1\\
\hline
    (6,2,0,0)  &0.4574, 1.4574 &15 &12 &0 &6 &0 &7 &6 &0 &3 &3 &8\\
\gr (2,0,0,2)  &0.5774 &20 &12 &4 &11 &6 &12 &8 &6 &6 &6 &9\\
    (10,0,0,2) &1      &16 &12 &0 &2 &5 &5 &6 &1 &2 &3 &7\\
    (0,2,2,0)  &1      &22 &12 &0 &8 &1 &8 &4 &2 &6 &4 &7\\
    (4,0,0,4)  &1.1547 &11 &12 &0 &4 &1 &3 &0 &2 &2 &2 &3\\
    (14,0,0,2) &1.2910 &8  &12 &0 &2 &0 &0 &0 &0 &3 &0 &2\\
    (0,2,6,0)  &1.2910 &12 &12 &0 &2 &0 &4 &3 &0 &2 &2 &5\\
    (8,0,0,4)  &1.2910 &10 &12 &0 &8 &3 &5 &1 &3 &1 &4 &4\\
    (6,0,0,6)  &1.7321 &6  &12 &0 &0 &0 &0 &0 &0 &3 &1 &0\\
    (0,2,10,0) &1.7321 &8  &12 &0 &0 &1 &1 &1 &0 &0 &1 &3\\
\hline
    (2,0,4,2)  &0.1685, 1.1423 &15 &24 &0 &2 &2 &4 &0 &2 &0 &5 &9\\
    (12,2,2,0) &0.2918, 1.9786 &13 &24 &0 &7 &0 &8 &2 &1 &3 &2 &7\\
    (2,0,6,4)  &0.2918, 1.9786 &8  &24 &0 &0 &1 &3 &0 &2 &0 &1 &3\\
    (4,0,2,2)  &0.4253, 0.9051 &16 &24 &0 &7 &3 &15 &1 &1 &1 &9 &15\\
    (4,0,6,2)  &0.4574, 1.4574 &12 &24 &2 &11 &4 &11 &11 &5 &8 &8 &14\\
    (4,0,4,4)  &0.6246, 1.7156 &7  &24 &0 &0 &0 &0 &0 &0 &0 &2 &2\\
    (8,0,2,2)  &0.7171, 1.0735 &12 &24 &0 &2 &2 &5 &1 &1 &1 &4 &7\\
    (8,0,6,2)  &0.7366, 1.5676 &10 &24 &0 &6 &2 &2 &1 &4 &2 &1 &6\\
    (6,2,4,0)  &0.7366, 1.5676 &15 &24 &2 &18 &4 &16 &16 &6 &11 &10 &16\\
    (2,0,2,4)  &0.8337, 1.4041 &9  &24 &0 &2 &0 &2 &0 &0 &0 &1 &3\\
    (10,0,4,2) &0.8337, 1.4041 &9  &24 &0 &0 &1 &3 &0 &1 &0 &2 &4\\
    (8,0,4,4)  &0.8505, 1.8101 &8  &24 &0 &0 &0 &0 &0 &0 &0 &1 &1\\
    (12,2,6,0) &0.8671, 2.1405 &8  &24 &0 &0 &0 &1 &0 &0 &0 &1 &3\\
    (6,2,8,0)  &1.2420, 1.8594 &9  &24 &0 &2 &0 &1 &2 &0 &0 &1 &0\\
\hline
\end{tabular}

}
}
\end{table}

\newpage

\begin{table}[H]
\caption{Lists of nodes of coloring diagrams for mono-triple with side $\sqrt3$. Every cell represents a single node. Numbers mean (from top to bottom): ordinal node number, then numbers of neighbor nodes previously colored $1\dots4$ in the left and vertex coordinates in the form $(a, b, c, d)$ in the right.}
\label{nodes}

\vspace{2mm}

\begin{tabular}{cc}
$S$ & \includegraphics[scale=0.115]{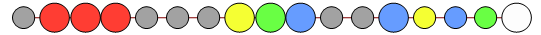}
\end{tabular} \par
\includegraphics[scale=0.337]{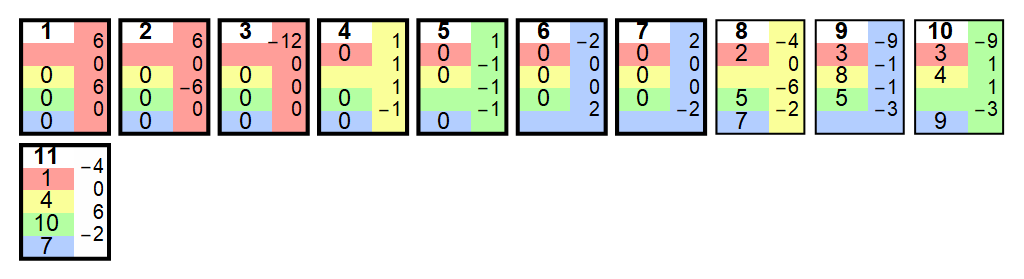}

\begin{tabular}{cc}
$L_1$ & \includegraphics[scale=0.115]{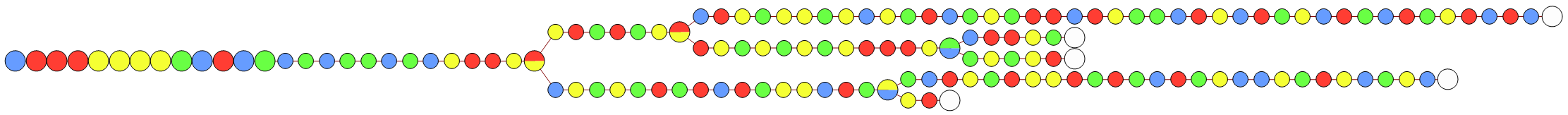}
\end{tabular} \par
\includegraphics[scale=0.337]{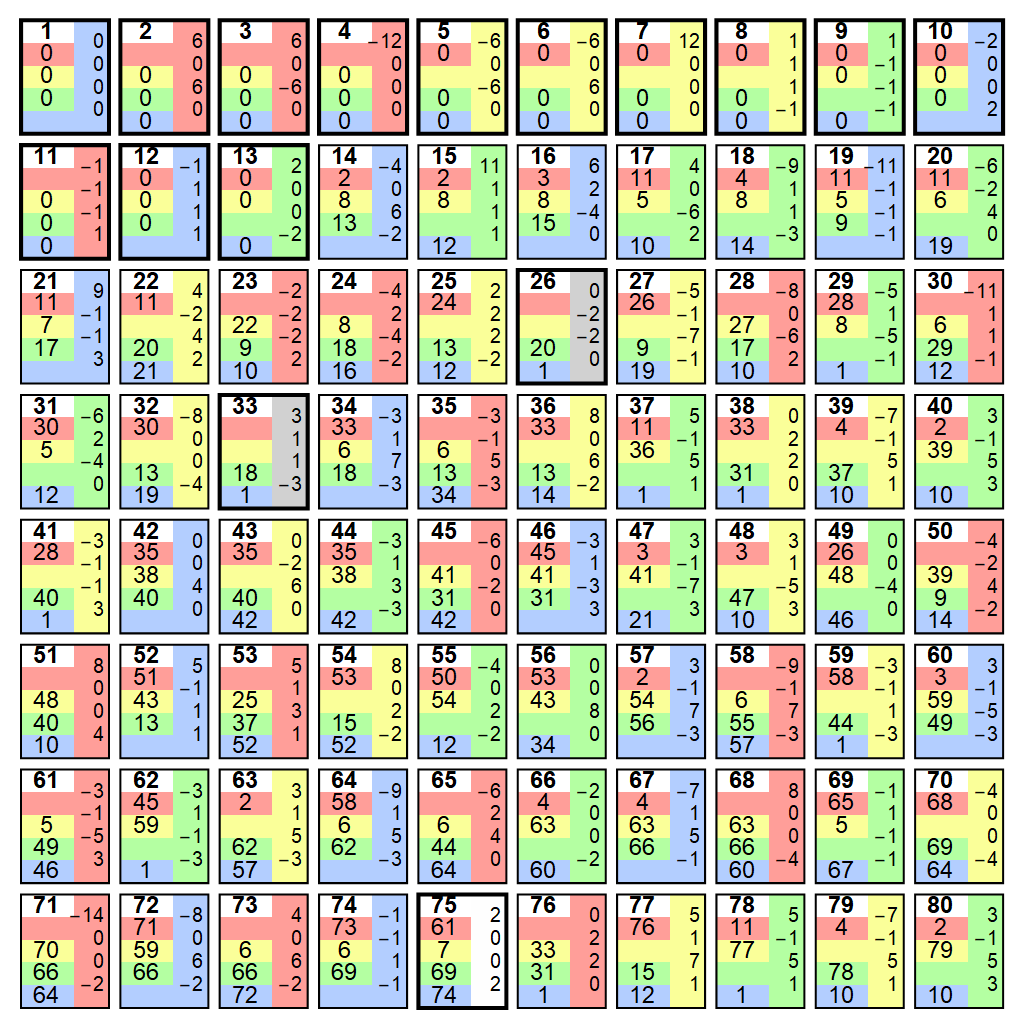}
\end{table}

\begin{figure}[H]
\includegraphics[scale=0.337]{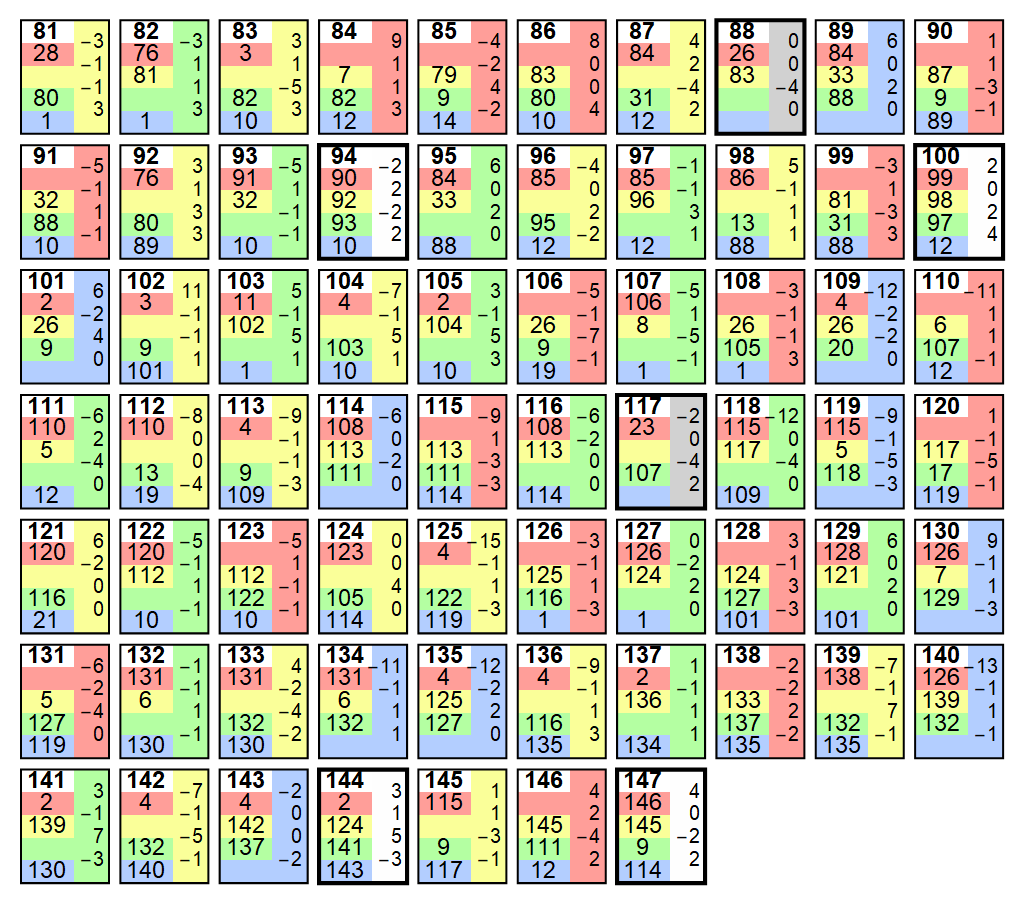}

\begin{tabular}{cc}
$L_2$ & \includegraphics[scale=0.115]{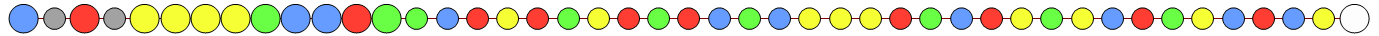}
\end{tabular} \par
\includegraphics[scale=0.337]{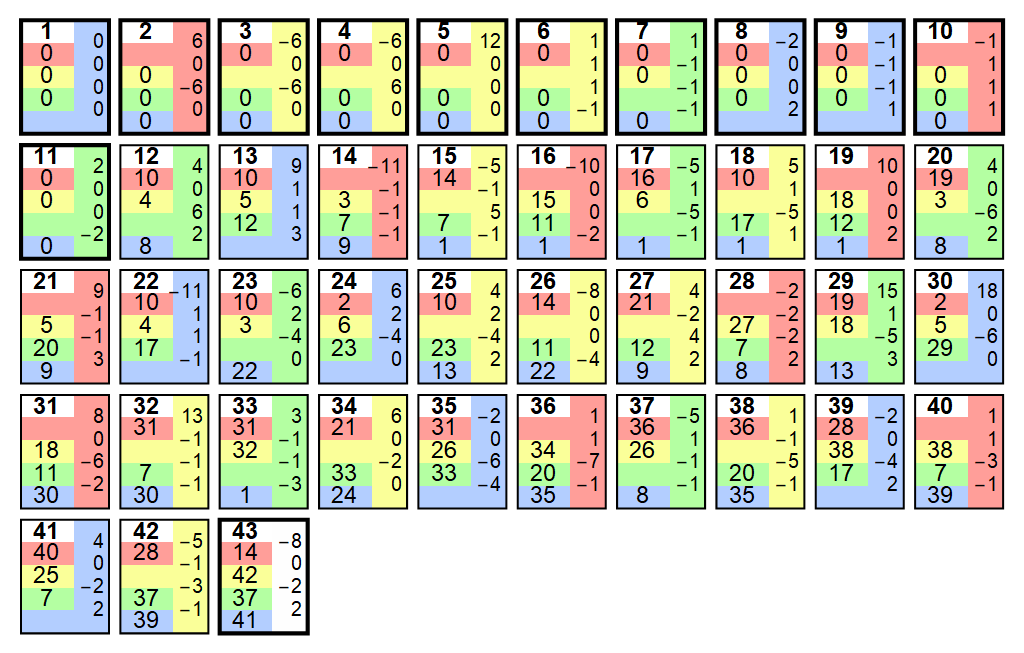}
\end{figure}

\begin{figure}[H]
\begin{tabular}{cc}
$L_3$ & \includegraphics[scale=0.115]{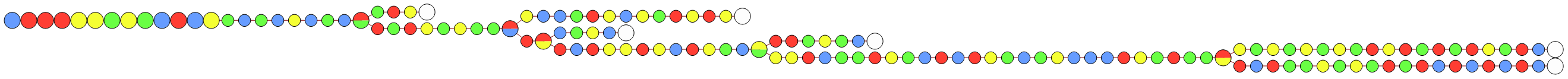}
\end{tabular} \par
\includegraphics[scale=0.337]{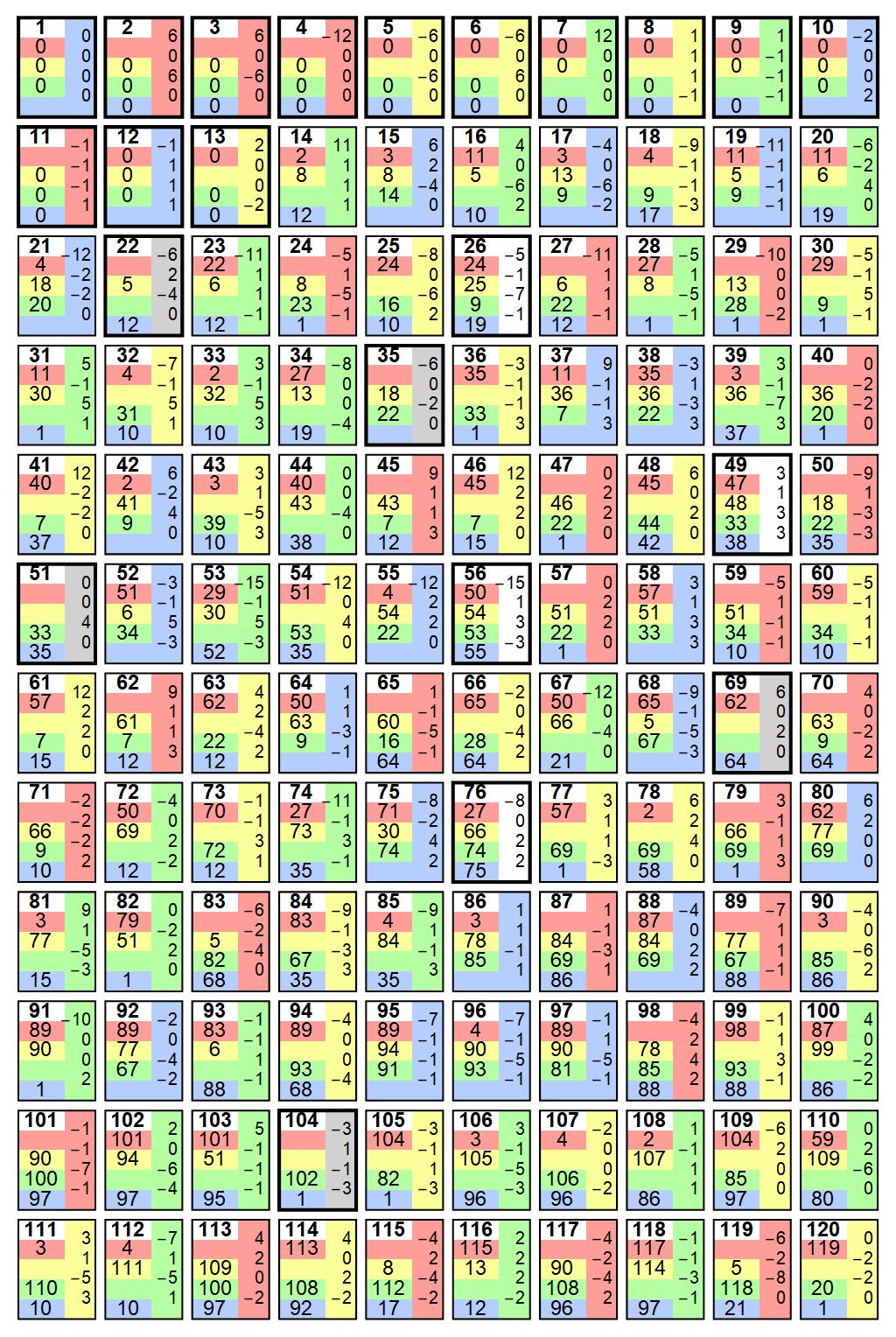}
\end{figure}

\begin{figure}[H]
\includegraphics[scale=0.337]{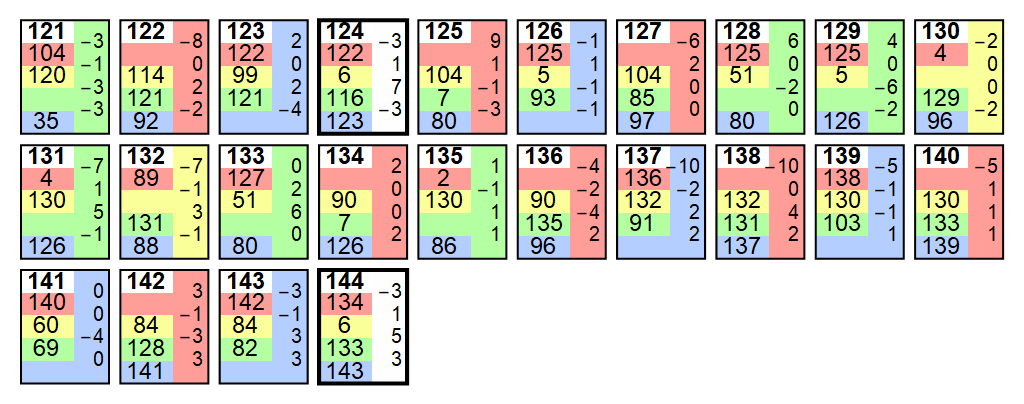}

\begin{tabular}{cc}
$L_4$ & \includegraphics[scale=0.115]{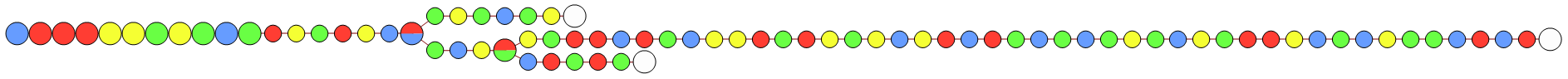}
\end{tabular} \par
\includegraphics[scale=0.337]{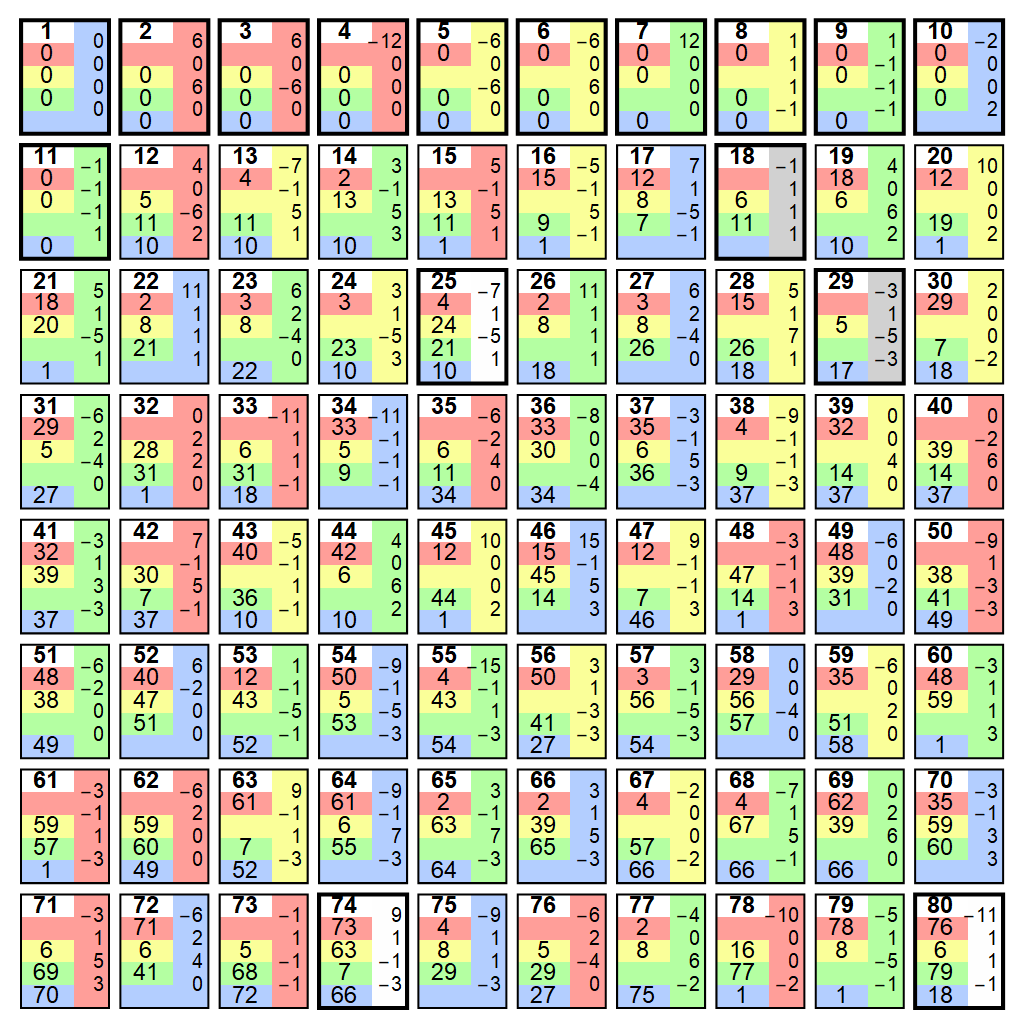}
\end{figure}

\begin{figure}[H]
\begin{tabular}{cc}
$L_5$ & \includegraphics[scale=0.115]{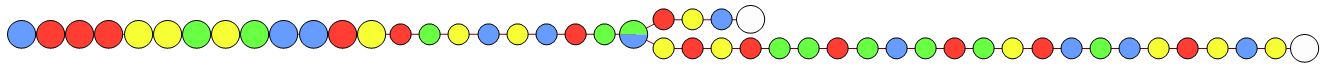}
\end{tabular} \par
\includegraphics[scale=0.337]{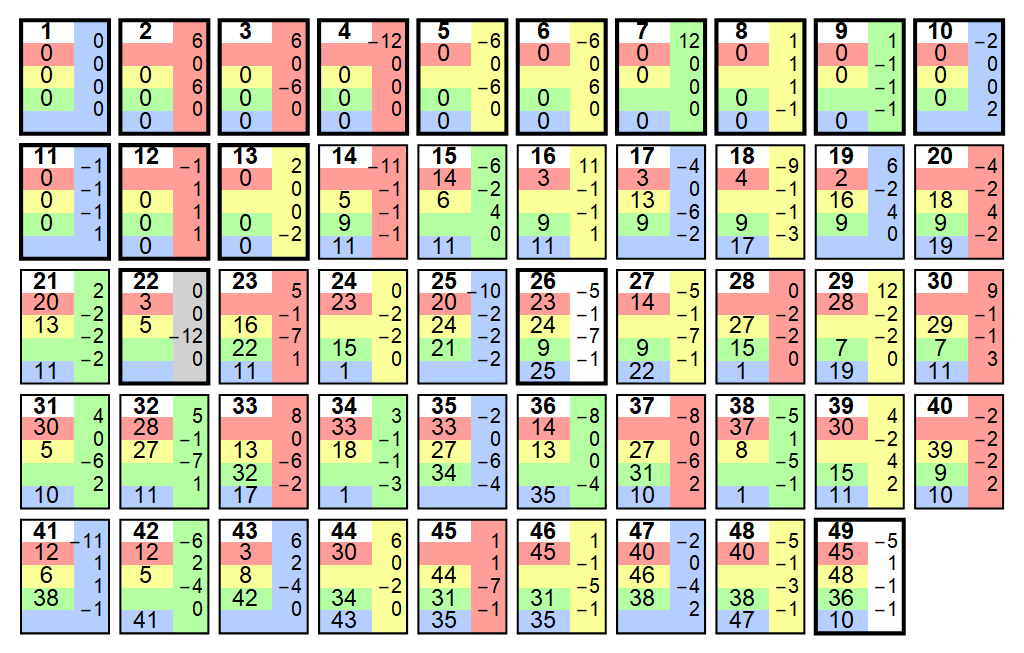}

\begin{tabular}{cc}
$L_6$ & \includegraphics[scale=0.115]{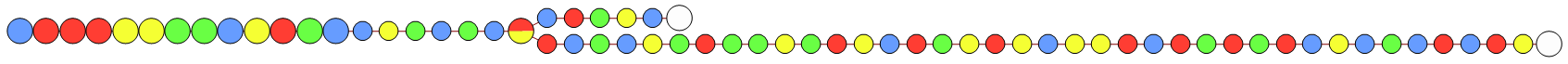}
\end{tabular} \par
\includegraphics[scale=0.337]{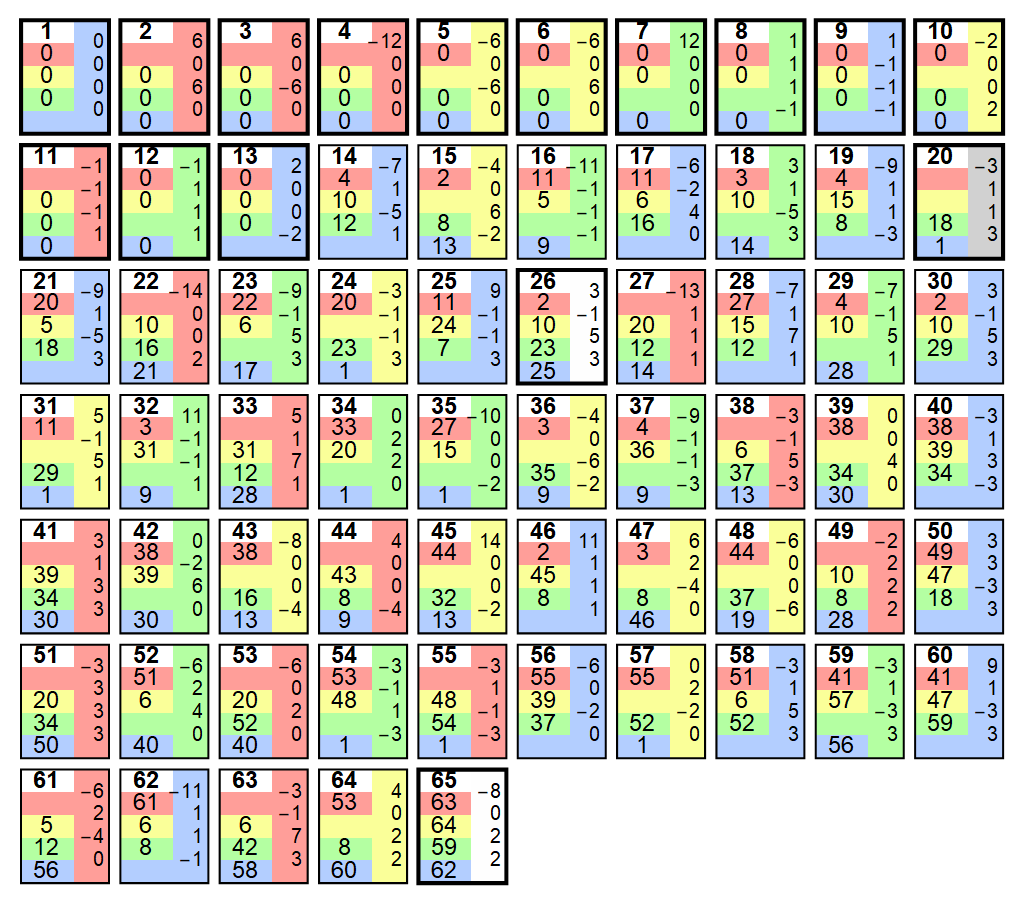}
\end{figure}

\begin{figure}[H]
\begin{tabular}{cc}
$L_7$ & \includegraphics[scale=0.115]{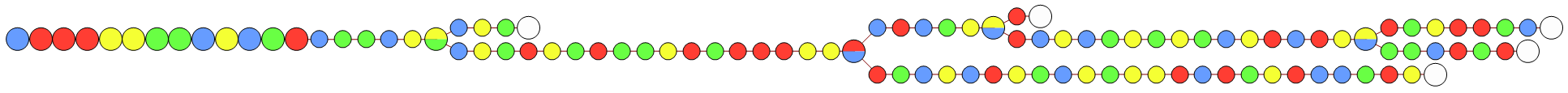}
\end{tabular} \par
\includegraphics[scale=0.337]{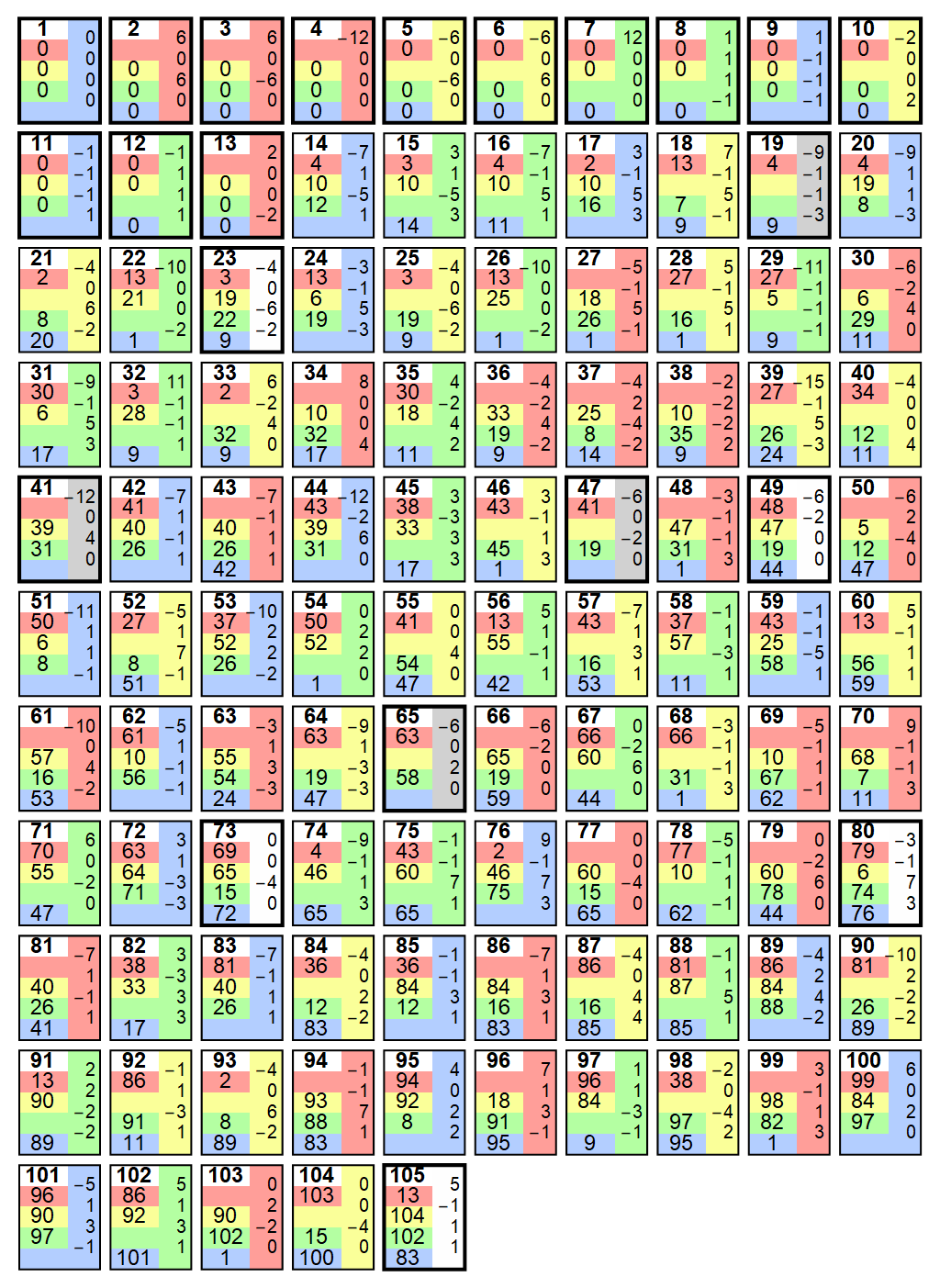}
\end{figure}

\begin{figure}[H]
\begin{tabular}{cc}
$L_8$ & \includegraphics[scale=0.115]{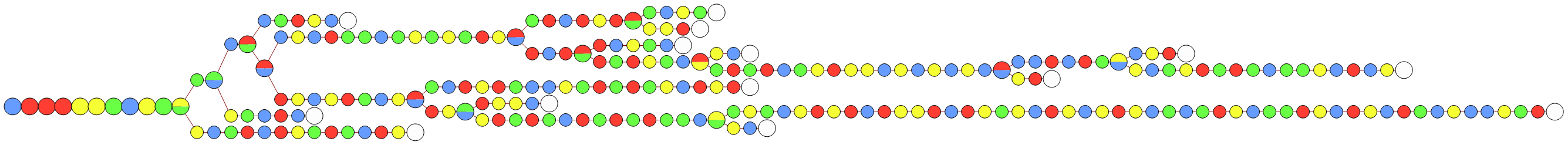}
\end{tabular} \par
\includegraphics[scale=0.337]{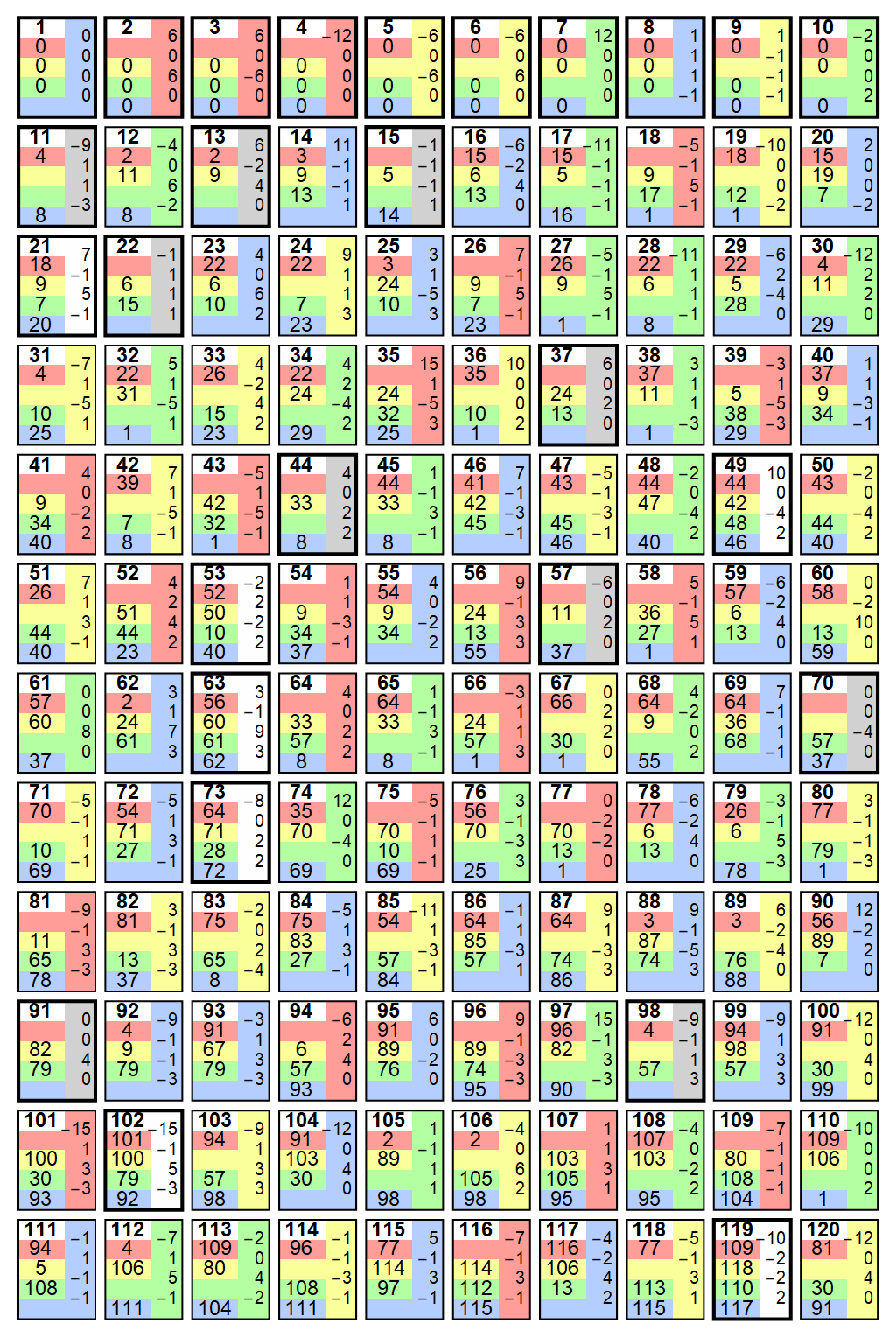}
\end{figure}

\begin{figure}[H]
\includegraphics[scale=0.337]{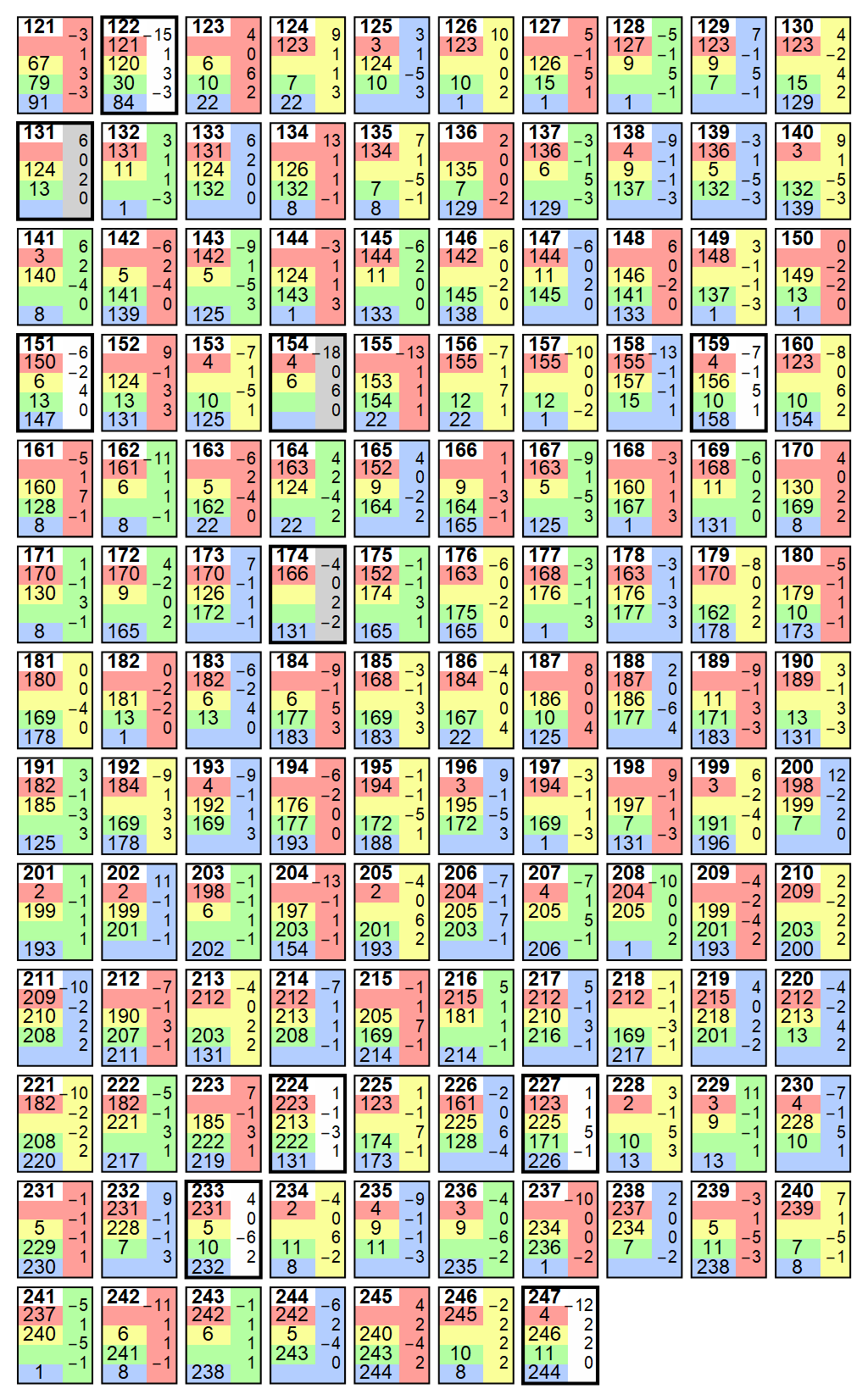}
\end{figure}

\section{Bonsai}

A few words about this Japanese art to produce small trees.

\paragraph{Growing a tree.}
The construction of a tree begins with the definition of a base graph $B$, which certainly contains a non-mono-triple, and a root coloring $R$ of some of its vertices. The requirements for the base graph are contradictory. The degree of the vertices and the order of the symmetry group of the base graph should be as large as possible. The larger the base graph, the more likely it can contain a smaller tree, but the more difficult it is to find. To resolve this contradiction, various working subgraphs $W$ are selected from the base graph $B$.

For each root coloring $R$ and a set of vertices of the working graph $W$ sorted in random order, a tree is constructed. The tree can be represented as a list of records, each of which reflects the act of assigning a color to a certain vertex and has four parameters: the number of the current node, the number of the next node, the number of the current vertex, and the color number of the current vertex. In the case of a branch, the coloring continues to the end, after which the next branch is created, while the branch node number is repeated with a different color.

The tree is built sequentially. At each step, a list of uncolored vertices is compiled, from which the first vertex with the minimum number of available colors is selected, optionally possessing some additional set of properties (for example, having the maximum degree). This vertex with all coloring options is added to the tree, after which the next step is performed. The construction of the tree is completed when all coloring paths have end nodes.

\paragraph{Thinning a tree.}
Each tree constructed is thinned out. This removes all redundant nodes that are not necessary to obtain the same set of end nodes. Before thinning, a description of tree nodes is formed, including for each color a list of nodes corresponding to either previously colored neighboring vertices of this color, or vertices that are colored later and cannot accept this color. Additionally, each node is characterized by its type (root, branch, stem, end) and status (necessary, undefined, redundant).

The status of a node depends on the number of its output connections with subsequent nodes and whether other outputs are connected to these nodes. A node is necessary if it is connected to some subsequent node such that no other nodes are connected to the same input color port. Branch and end nodes can also be referred to the necessary ones. A node is redundant if it is not connected to any subsequent nodes. The status of the remaining nodes is undefined.

First all redundant nodes are removed. This is done in several iterations, as the chains of redundant nodes may appear. Then undefined nodes are removed until only the necessary ones remain. The order in which the next deleted node is selected may depend on the location of nodes in the tree, the length of the corresponding branch, the number of connected nodes, and the vertex degree. We were thinning by going backwards from the end to the beginning of the tree. First of all, nodes with fewer output connections were removed. Root and branch nodes were not removed. After deleting each next node, the status of the remaining nodes may change. One can try several different thinning options and then choose the best tree with the fewest remaining nodes.
We found that, on average, thinning reduces the number of nodes by half.

\paragraph{Tree minimization.}
The procedure for constructing and thinning trees is repeated many times, and the smallest of the resulting trees is left. At the same time, one can change the set of working graphs, the order of vertex numbering, the function of selecting a vertex when building a tree. Elements of genetic algorithms can be used. For example, the best tree is taken as a basis, relative to which small random changes (mutations) are made until this leads to improvement, after which a new tree is taken as a basis and the search is repeated.

We started with a set of a dozen working graphs of different sizes from 259 to 727 vertices, selected at random from the set of graphs left over from the previous study (minimizing a 5-chromatic graph). After a series of tree minimization stages, all trees began to fit into the 481-vertex graph shown in Fig.\ref{base}. Further minimization was carried out on the subgraphs of this base graph. In the last stages, we discarded different parts of the tree and tried to grow new ones instead.

\section{Back to the future}
One can set the task of obtaining the smallest possible coloring tree. The simplest and most convenient optimization criterion is the total number of nodes. One can penalize additionally for each branch, but usually such a penalty is implicitly already included in the good solution.

It is likely that the total number of nodes could be reduced by another hundred or two. Among other things, the following obvious directions of attack can be indicated: deeper search, expansion of the base graph, various combinations of orbits, other root graphs, other sets of trees, development of a genetic approach (using in parallel a set of surviving descendants, permutation of loci). 
Note that in our case the 4-Golomb graph and other variants of the 2-Golomb graphs as a root one did not give a good result.

The proof presented here may leave a sediment of dissatisfaction. We tend to explain this by the fact that it is expected not only \textit{human-verifiable} proof, but really \textit{human} proof. The difference is that in the latter case, the proof is created by a person, while in the first case it is not necessary: it is only important that the person is able to complete the verification in a reasonable number of steps.

It is believed that the average person cannot keep more than 7 objects in the field of attention. This highly limits our arsenal of proof methods. The less redundant a mathematical object is, the more difficult it is to deduce any consequences from it, especially if it is complex enough, that is, it has many different internal relationships. Perhaps 5-chromatic graphs are just such objects.

\section{Thank you}
Aubrey, David, Nicholas, Jasper, and Alexander.

\end{document}